%% file: 170722Poisson.tex
\documentclass[a4paper,11pt]{amsart} 
\usepackage{amsmath,amsxtra,amssymb,latexsym, amscd,amsthm}
\usepackage[mathscr]{eucal}
\usepackage{mathrsfs}
\usepackage{bbm}
\usepackage{enumerate}
\usepackage{pict2e}
\usepackage{tikz}
\usepackage{graphicx}
\usepackage[a4paper]{geometry}
\usepackage{color}
\usepackage{hyperref}
\numberwithin{equation}{section}
\theoremstyle{plain}
\newtheorem{thm}{Theorem}[section]
\newtheorem{prp}[thm]{Proposition}

\newtheorem{lem}[thm]{Lemma}
\newtheorem*{euc*}{Euclidean division}
\newtheorem*{fek*}{Fekete's Lemma}
\newtheorem*{kin*}{Kingman's Subadditive Ergodic Theorem}
\newtheorem*{fur*}{Furstenberg-Kesten Theorem}
\theoremstyle{definition}

\newtheorem{rem}[thm]{Remark}

\newtheorem*{rem*}{Remark}

\newcommand{\dd}{\mathrm{d}}

\renewcommand{\Im}{\operatorname{Im}}
\renewcommand{\Re}{\operatorname{Re}}

\makeatletter
\newcommand*{\ov}[1]{%
  $\m@th\overline{\mbox{#1}}$%
}
\newcommand*{\ovA}[1]{%
  $\m@th\overline{\mbox{#1}\raisebox{3mm}{}}$%
}
\newcommand*{\ovB}[1]{%
  $\m@th\overline{\mbox{#1\rule{0pt}{3mm}}}$%
}
\newcommand*{\ovC}[1]{%
  $\m@th\overline{\mbox{#1\strut}}$%
}
\newcommand*{\ovD}[1]{%
  $\m@th\overline{\mbox{#1\vphantom{\"A}}}$%
}
\newcommand*{\ovE}[1]{%
  $\m@th\overline{\raisebox{0pt}[1.2\height]{#1}}$%
}
\newcommand*{\ovF}[1]{%
  $\m@th\overline{\raisebox{0pt}[\dimexpr\height+0.3mm\relax]{#1}}$%
}
\newcommand*{\ovG}[1]{%
  $\m@th\overline{\raisebox{0pt}[\dimexpr\height+1mm\relax]{#1\vphantom{A}}}$%
}
\makeatother

\newcommand{\Z}{\mathbb{Z}}

\newcommand{\R}{\mathbb{R}}
\newcommand{\C}{\mathbb{C}}

\DeclareMathOperator{\tr}{tr}

\DeclareSymbolFont{extraup}{U}{zavm}{m}{n}
\DeclareMathSymbol{\varheart}{\mathalpha}{extraup}{86}
\DeclareMathSymbol{\vardiamond}{\mathalpha}{extraup}{87}

\title{Poisson kernel expansions for Schr\"odinger operators on trees}
\author{Nalini \textsc{Anantharaman} and Mostafa \textsc{Sabri}}
\address{Universit\'e de Strasbourg, CNRS, IRMA UMR 7501, F-67000 Strasbourg, France.}
\email{anantharaman@math.unistra.fr}
\address{Universit\'e de Strasbourg, CNRS, IRMA UMR 7501, F-67000 Strasbourg, France.}
\address{Department of Mathematics, Faculty of Science, Cairo University, Cairo 12613, Egypt.}
\email{sabri@math.unistra.fr}
\usepackage{calc}
\usepackage{yhmath}
\usepackage{graphicx}

\makeatletter
\newlength{\temp@wc@width}
\newlength{\temp@wc@height}
\newcommand{\widecheck}[1]{%
  \setlength{\temp@wc@width}{\widthof{$#1$}}%
  \setlength{\temp@wc@height}{\heightof{$#1$}}%
  #1\hspace{-\temp@wc@width}%
  \raisebox{\temp@wc@height+2pt}[\heightof{$\widehat{#1}$}]%
     {\rotatebox[origin=c]{180}{\vbox to 0pt{\hbox{$\widehat{\hphantom{#1}}$}}}}%
}
\makeatother
\input{def.tex}

\begin{document}

\begin{abstract}
We study Schr\"odinger operators on trees and construct associated Poisson kernels, in analogy to the laplacian on the unit disc. We show that in the absolutely continuous spectrum, the generalized eigenfunctions of the operator are generated by the Poisson kernel. We use this to define a ``Fourier transform'', giving a Fourier inversion formula and a Plancherel formula, where the domain of integration runs over the energy parameter and the geometric boundary of the tree.
\end{abstract}

\maketitle

\section{Introduction}                     \label{sec:intropoi}
\subsection{Presentation of results}In this note we are interested in the spectral theory of discrete Schr\"odinger operators on trees. Our main purpose is to use the simple combinatorics of paths on trees to understand better the geometric structure of generalized eigenfunctions.

Let $\mathcal{T}$ be a tree with a uniformly bounded degree. By some abuse of notation, we also denote its vertex set by $\mathcal{T}$. We study a Schr\"odinger operator $H$ on $\mathcal{T}$ given by
\[
H = \mathcal{A}+V \, ,
\]
where $\mathcal{A}$ is the adjacency matrix
\[
(\mathcal{A}\psi)(v) = \sum_{w\sim v} \psi(w)
\]
and $V: \mathcal{T}\To \IR$ is a real-valued potential, so that $H$ is self-adjoint on its domain
\[
D(H) = \{\psi \in \ell^2(\mathcal{T}) : V\psi \in \ell^2(\mathcal{T})\} \, .
\]
Here, $w\sim v$ means that $w$ and $v$ are nearest neighbors.

Let us give some background on the theory of generalized eigenfunctions.

If $F$ is a bounded Borel function on $\R$, we know by the spectral theorem that for any $v,w\in \mathcal{T}$, we may find a Borel measure $\mu_{v,w}$ on $\R$ such that
\begin{equation}             \label{eq:decompo1}
F(H)(v,w) = \int_{\R} F(E) \,\dd \mu_{v,w}(E) \, .
\end{equation}
Here $F(H)(v,w)=\langle \delta_v,F(H)\delta_w\rangle$ is the matrix of $F(H)$ in the basis $\{\delta_v\}_{v\in\mathcal{T}}$.
The theory of generalized eigenfunction expansions refines this expression by constructing a spectral measure $\rho_H$ on $\IR$ and functions $Q_{E,w}: \mathcal{T} \To \IC$ satisfying $HQ_{E,w} = EQ_{E,w}$, such that
\begin{equation}             \label{eq:decompo2}
F(H)(v,w) = \int_{\R} F(E) Q_{E,w}(v)\,\dd \rho_H(E) \, .
\end{equation}
See \cite[Section 7]{Kir} and \cite[Chapter 15]{BSU} for details. Such an expansion proved to be useful in the context of Anderson localization when $H_{\omega}$ is a random Schr\"odinger operator. In fact, pure point spectrum in an interval $I$ will follow if one shows that for $\rho_H$-a.e. $E\in I$, the function $Q_{E,w}$ lives in $\ell^2(\mathcal{T})$. Expression (\ref{eq:decompo2}) is also used to estimate the Hilbert-Schmidt norms that arise in the study of dynamical localization; see \cite[Lemma 4.1]{GK}.

Efforts have been made to push the expansion further. In \cite[Section 15.3]{BSU} and \cite{LT}, the authors abstractly construct functions $\varphi_{E,j}:\mathcal{T}\To \C$, $j=1,\dots,N_E \le \infty$, such that $H\varphi_{E,j}=E\varphi_{E,j}$ for $\rho_H$-a.e. $E$, and
\begin{equation}         \label{eq:decompo3}
F(H)(v,w) = \int_{\R} F(E) \sum_{j=1}^{N_E} \varphi_{E,j}(v)\overline{\varphi_{E,j}(w)}\,\dd \rho_H(E) \, .
\end{equation}
In other words, $Q_{E,w}(v) = \sum_{j=1}^{N_E} \varphi_{E,j}(v)\overline{\varphi_{E,j}(w)}$ for $\rho_H$-a.e. $E$. The importance of this expression is that it allows to define an abstract Fourier transform by
\[
\widehat{f}_j(E) = \langle \varphi_{E,j},f\rangle_{\C^\mathcal{T}} = \sum_{w\in \mathcal{T}} \overline{\varphi_{E,j}(w)} f(w)
\]
for functions $f:\mathcal{T}\to \C$ of finite support. The functions $\varphi_{E,j}$ thus play the role of the ``plane waves'' for the euclidean laplacian. As a consequence of (\ref{eq:decompo3}), one obtains a Fourier inversion formula
\[
f(v) = \int_{\R} \sum_{j=1}^{N_E} \widehat{f}_j(E) \,\varphi_{E,j}(v)\,\dd \rho_H(E) \, ,
\]
and a Plancherel formula, namely if $f,g:\mathcal{T}\to \C$ have finite support, then
\begin{equation}          \label{planchabstract}
\langle f,g \rangle_{\ell^2(\mathcal{T})} = \int_{\R}\sum_{j=1}^{N_E} \overline{\widehat{f}_j(E)}\widehat{g}_j(E)\,\dd \rho_H(E) \, .
\end{equation}
The Plancherel formula can be extended by continuity to all $f,g\in \ell^2(\mathcal{T})$. Moreover, the previous expansions are actually valid for general self-adjoint operators on abstract Hilbert spaces $\mathscr{H}$ (see \cite{BSU}).

In this paper we show that for particular models, one can obtain expansion formulas which are very explicit. Our approach is totally different, it uses a direct geometric analysis of the Green function, and is inspired by an existing analogy between the adjacency matrix on the tree and the laplacian on the unit disc. 

The functions $\varphi_{E,j}$ in our case are replaced by explicit functions $P_{E,\xi}$, which we call Poisson kernel. The name ``Poisson kernel'' is borrowed from the potential theory of the unit disc, which we briefly recall in Section~\ref{sec:backpois}. The parameter $\xi$ runs over the geometric boundary $\partial \mathcal{T}$ of the tree. In Proposition \ref{p:exists}, we first establish the existence of the Poisson kernel for Lebesgue-a.e. $E\in \R$, and show that $HP_{E,\xi} = EP_{E,\xi}$. Next, assuming the Schr\"odinger operator $H$ has purely absolutely continuous spectrum in some measurable set $I$, we construct an explicit positive measure $\nu_E$ on $\partial \mathcal{T}$ such that
\begin{equation}             \label{eq:decompo0}
F(H)(v,w) = \int_I \int_{\partial \mathcal{T}} F(E) P_{E,\xi}(v) \overline{P_{E,\xi}(w)}\,\dd \nu_E(\xi)\,\dd E
\end{equation}
for any bounded Borel function $F:I\to \C$.

The assumption of absolutely continuous spectrum is known to hold for example if $H=\mathcal{A}$ on $\mathcal{T}$, for many trees $\mathcal{T}$ of finite cone type (in particular, if $\mathcal{T}$ is a regular tree). In this case, the spectrum is purely absolutely continuous. In fact, the authors in \cite{KLW} establish more generally that the absolutely continuous spectrum of $\mathcal{A}$ remains stable under small radially symmetric perturbations $H=\mathcal{A} + \lambda V$, if the tree is non-regular. Large parts of the absolutely continuous spectrum also remain stable under small random perturbations $H_{\omega} = \mathcal{A}+\lambda V_{\omega}$; see \cite{KLW2}. For example, in the particular case of $(q+1)$-regular trees, it is shown in \cite{Klein, FHS} that the Anderson model has purely absolutely continuous spectrum almost surely in any interval $I=[-E_0,E_0] \subset (-2\sqrt{q},2\sqrt{q})$, if the disorder $\lambda$ is small enough. The results of Aizenman and Warzel \cite{AW} go further, by showing existence of absolutely continuous spectrum outside $(-2\sqrt{q},2\sqrt{q})$. Our results thus apply to these models.

As in the previous discussion, we define the Fourier transform
\[
\widehat{f}_{\xi}(E) = \langle P_{E,\xi}, f \rangle_{\C^\mathcal{T}}
\]
for $f:\mathcal{T}\to \C$ of finite support. The content of Theorem \ref{thm:planch} is an inversion formula
\[
[F(H)f](v) = \int_I \int_{\partial \mathcal{T}} F(E) \widehat{f}_{\xi}(E) P_{E,\xi}(v)\,\dd \nu_E(\xi)\,\dd E
\]
which implies a Plancherel formula
\[
\langle f,F(H)g\rangle = \int_I \int_{\partial \mathcal{T}} F(E)\overline{\widehat{f}_{\xi}(E)}\widehat{g}_{\xi}(E)\,\dd \nu_E(\xi)\,\dd E \,,
\]
for any $f$ and $g$ on $\mathcal{T}$ of finite support.

In Theorem \ref{thm:gef}, we obtain a representation formula for eigenfunctions of the Schr\"odinger operator $H$ by integrals of the Poisson kernel over the boundary. This is valid for complex eigenvalues $\gamma\in \IC\setminus \IR$, as well as for almost-every real eigenvalue in the absolutely continuous spectrum (the associated eigenfunctions are necessarily not in $\ell^2$).

The analogy between the spectral theory on regular trees and on the unit disc was first put forward in the influential paper \cite{Cartier}. There, seeing the tree as a Cayley graph for a free group, the author obtains an isomorphism between the space of harmonic functions (i.e. solutions of $\mathcal{A}f=0$), and a space of distributions on $\partial\mathcal{T}$. Our work builds on previous constructions in \cite{Chouc, FTP, FTN, Ao, FTS}, where expansions in Poisson kernel are proved for $H=\mathcal{A}$ on a regular tree, and for anisotropic random walks on the free product $\Z/2\Z \star \dots \star \Z/2\Z$. In those situations, the tree has a homogeneous structure, and the results can be used to understand the unitary representations of the automorphism groups of the tree and of the group $\IZ/{2\IZ}\star \IZ/{2\IZ}\star\ldots \star \IZ/{2\IZ}$.
This work arose from the remark that the aforementioned theories may be extended to more general Schr\"odinger operators on trees. However, since those trees have no homogeneous structures, no representation theory will be involved.

The Plancherel formula in our Theorem~\ref{thm:planch} implies that
\[
\|F(H)K\|_{HS}^2 = \int_I \int_{\partial \mathcal{T}} |F(E)|^2 \|KP_{E,\xi}\|_{\ell^2(\cT)}^2\,\dd \nu_E(\xi)\,\dd E
\]
for any $K$ on $\mathcal{T}\times \mathcal{T}$ of finite support. This provides a convenient formula to estimate Hilbert-Schmidt norms, as the Poisson kernel plays an interesting geometric role. In fact, the first named author already used it in \cite{A} in the study of quantum ergodicity for homogeneous and anisotropic random walks on regular trees.

\begin{rem}
The results of this paper generalize without difficulty to self-adjoint operators of the form $(H_p\psi)(v) = \sum_w p_v(w) f(w)$, where $p_v(w) = 0$ if $d(v,w)>1$, assuming all coefficients $p_v(w)$ are real, with $p_v(w)=p_w(v)$ and $p_v(w)\neq 0$ whenever $v\sim w$. For more details see Remark \ref{r:more}.
\end{rem}

\subsection{Background on Poisson kernels}               \label{sec:backpois}
The word ``Poisson kernel'' is traditionally used in the potential theory of the $2$-dimensional disc $\ID=\{ z=x+iy\in \C, |z|<1\}$. In this context, the Poisson kernel is a family of functions parametrized by the boundary of the disc
$\partial \ID=\{ z=x+iy\in \C, |z|=1\}$. For $\omega\in \partial \ID$ and $z\in \ID$, we let
\[
P_\omega(z)=\frac{1-|z|^2}{|z-\omega|^2} \,.
\]
For all $\omega$, $P_\omega$ is a solution to $\Delta P_\omega=0$, where $\Delta=\partial_x^2+\partial_y^2$ is the euclidean laplacian. The Poisson kernel is useful to solve the Laplace problem on the disc~: if $f$ is an integrable function on $\partial \ID$ for the Lebesgue measure, and if we put
\[
(Pf)(z)=\int_{\partial \ID} P_\omega(z) f(\omega)d\omega \,,
\]
then $Pf$ is a solution to $\Delta (Pf)=0$. Moreover, for almost-all $\omega\in \partial \ID$, the limit of $Pf(z)$ as $z$ tends to a boundary point $\omega$ (in a nontangential way) is equal to $f(\omega)$. If we started with a continuous $f$, then this is true for {\em all} $\omega$ (see for instance \cite[Chapter 11]{Rud}).

If, instead of the euclidean laplacian on $\ID$, we consider the hyperbolic laplacian
\[
\Delta_{hyp}=\frac{(1-|z|^2)^2}{4}(\partial_x^2+\partial_y^2) \,,
\]
the picture is even more complete. Define now, for $s\in\IC$,
\[
P_{\omega, s}(z)=\left(\frac{1-|z|^2}{|z-\omega|^2}\right)^s \,.
\]
We have $\Delta_{hyp}P_{\omega, s}=-s(1-s)P_{\omega, s}$, that is to say, the functions $P_{\omega, s}$ are eigenfunctions of the hyperbolic laplacian. In this context, the fundamental work by Helgason \cite{Hel} gives
\begin{itemize}
\item an integral representation theorem for arbitrary eigenfunctions~: ``every eigenfunction of eigenvalue $-s(1-s)$ can be represented in a unique way by integrating the Poisson kernel $P_{\omega, s}$ against some analytic functional over the boundary $\partial \ID$ '' (\cite{Hel}, Theorem 4.3);
\item a ``Fourier transform'' allowing to represent compactly supported functions as superpositions of the Poisson kernels $P_{\omega, s}$, summed over $\omega\in \partial\ID$ and over the spectral parameter $s$. The Poisson kernels are constant on horocycles and thus play the role of ``plane waves'' (\cite{Hel}, Theorem 4.2 (i));
\item a Plancherel formula expressing the $L^2$-norm of a function in terms of its Fourier representation  (\cite{Hel}, Theorem 4.2 (ii)).
\end{itemize}

%

{\bf Open questions.} It would be interesting to ask if our construction of the ``Poisson kernel'' for Schr\"odinger operators on trees can be extended to other types of graphs, assumed for instance to be Gromov-hyperbolic. In the region of ``positive spectrum'' (meaning the region of existence of positive eigenfunctions), the question has been extensively studied, with some very recent remarkable advances. In that context, what we called ``Poisson kernel'' bears the name ``Martin kernel'', and allows to represent all {\em non-negative} eigenfunctions by an integral of the kernel over the boundary. For Gromov-hyperbolic graphs, the coincidence of the Martin boundary with the geometric boundary has been proven in \cite{Anc1, Anc2} in the interior of the positive spectrum, in \cite{GL, G} at the top of the positive spectrum, which coincides with the bottom of the $\ell^2$-spectrum. This relies highly on the fact that we are in a region where the Green function is positive. Inside the $\ell^2$-spectrum, not much is known in general. Our construction relies on the existence of absolutely continuous spectrum, but except for trees, no examples of Gromov-hyperbolic graphs with absolutely continuous spectrum seem to be known (see \cite{KLW3} for a more detailed discussion).

A similar question arises in the case
of the spectral theory of the laplacian on $\widetilde M$, the universal cover of a compact negatively curved surface $M$.  If $M$ has constant curvature $-1$, then $\widetilde M$ is isometric to the hyperbolic disc, the spectrum is purely absolutely continuous, and the spectral theory is completely described by the Helgason-Fourier transform described above \cite{Hel}. However, if the curvature is variable, not much seems to be known. In the positive spectrum, a Martin kernel can be constructed, and allows for integral representation of all positive eigenfunctions by integrating the ``Martin kernel'' over the geometric boundary \cite{Anc2, LL}. But again, the nature of the $\ell^2$-spectrum is not known, which prevents from going further towards a theory of Poisson kernels in that part of the spectrum.

\section{The Green function on the tree}

Given $v,w \in \mathcal{T}$ with $v \sim w$, we denote by $\mathcal{T}^{(v|w)}$ the tree obtained by removing from $\mathcal{T}$ the branch emanating from $v$ that passes through $w$. We keep the vertex $v$, so $v\in \mathcal{T}^{(v|w)}$.

We define the restriction $H^{(v|w)}(u,u') = H(u,u')$ if $u,u'\in \mathcal{T}^{(v|w)}$ and zero otherwise. The Green functions are denoted by 
\[
G(u,u';\gamma) = \langle \delta_u,(H-\gamma)^{-1}\delta_{u'}\rangle \quad \text{and} \quad G^{(v|w)}(u,u';\gamma) = \langle \delta_u, (H^{(v|w)}-\gamma)^{-1}\delta_{u'}\rangle
\]
for $\gamma$ in the resolvent set of $H$ and $H^{(v|w)}$, respectively.

Recall that for any $v\in \mathcal{T}$ and $\gamma \in \C \setminus \R$, we have
\begin{equation}               \label{eq:green1}
G(v,v;\gamma) = \frac{1}{V(v) - \gamma - \sum_{u \sim v} G^{(u|v)}(u,u;\gamma)} \, .
\end{equation}
If $v,w \in \mathcal{T}$ and $v \sim w$, we also have
\begin{equation}            \label{eq:green2}
G^{(v|w)}(v,v;\gamma) = \frac{1}{V(v) - \gamma - \sum_{u \in \mathcal{N}_v \setminus \{w\}} G^{(u|v)}(u,u;\gamma)} \, ,
\end{equation}
where $\mathcal{N}_v = \{u:u\sim v\}$. These identities are well-known and follow from the resolvent identity; see \cite[Proposition 2.1]{Klein} for a proof.

If $(v_0,\ldots,v_k)$ is a non-backtracking path in $\mathcal{T}$ and $\gamma \in \C \setminus \R$, we have
\begin{equation}                  \label{eq:multigreen}
G(v_0,v_k;\gamma) = (-1)^k \prod_{j=0}^{k-1} G^{(v_j|v_{j+1})}(v_j,v_j;\gamma) \cdot G(v_k,v_k;\gamma) \, .
\end{equation}
This is also well-known; see \cite[Chapter 1]{FTS} or \cite[Equation (2.8)]{Klein} and use induction.

Given $\gamma \in \C \setminus \R$, we denote
\[
G(v,v;\gamma) = \frac{-1}{2m_v^{\gamma}} \quad \text{and} \quad \zeta_w^{\gamma}(v) = -G^{(v|w)}(v,v;\gamma) \, .
\]

\begin{lem}                     \label{lem:zetapot}
For any $v \in \mathcal{T}$ and $\gamma = E+i\eta  \in \C^+=\{z\in \IC, \Im z >0\}$, we have
\begin{equation}               \label{eq:green3}
\gamma = V(v) + \sum_{u \sim v} \zeta_v^{\gamma}(u)+2m^{\gamma}_v \quad \text{and} \quad \gamma = V(v) + \sum_{u \in \mathcal{N}_v \setminus \{w\}} \zeta_v^{\gamma}(u) + \frac{1}{\zeta_w^{\gamma}(v)} \, .
\end{equation}
For any non-backtracking path $(v_0,\dots,v_k)$ in $\mathcal{T}$,
\begin{equation}             \label{eq:multigreen2}
G(v_0,v_k;\gamma) = \frac{-\prod_{j=0}^{k-1} \zeta_{v_{j+1}}^{\gamma}(v_j)}{2m^{\gamma}_{v_k}} \, ,
\end{equation}
\begin{equation}          \label{eq:multigreen3}
G(v_0,v_k;\gamma) = \zeta_{v_1}^{\gamma}(v_0) G(v_1,v_k;\gamma) = \zeta_{v_{k-1}}^{\gamma}(v_k) G(v_0,v_{k-1};\gamma) \, ,
\end{equation}
\begin{equation}           \label{eq:multigreen4}
G^{(v_k|v_{k+1})}(v_0,v_k;\gamma) = - \prod_{j=0}^k \zeta_{v_{j+1}}^{\gamma}(v_j) \quad \text{and} \quad G^{(v_1|v_0)}(v_1,v_k;\gamma) = - \prod_{j=0}^{k-1} \zeta_{v_j}^{\gamma}(v_{j+1}) \, .
\end{equation}
Also, for any $w\sim v$, we have
\begin{equation}           \label{eq:mv}
\zeta_w^{\gamma}(v) = \frac{m_w^{\gamma}}{m_v^{\gamma}} \,\zeta_v^{\gamma}(w) \quad \text{and} \quad 
\frac{1}{\zeta_w^{\gamma}(v)} - \zeta_v^{\gamma}(w) = 2m^{\gamma}_v \, .
\end{equation}
For any $v,w\in \mathcal{T}$, we have
\begin{equation}              \label{eq:greensym}
G(v,w;\gamma) = G(w,v;\gamma) \, .
\end{equation}
Next,
\begin{equation}             \label{eq:sumzeta}
\sum_{u\in \mathcal{N}_v\setminus \{w\}} |\Im \zeta_v^{\gamma}(u)| = \frac{|\Im \zeta_w^{\gamma}(v)|}{|\zeta_w^{\gamma}(v)|^2} - \eta \, .
\end{equation}
Finally, if $\Psi_{\gamma,v}(w) = \frac{1}{\pi}\Im G(v,w;\gamma)$, then for any path $(v_0,\dots,v_k)$ in $\mathcal{T}$,
\begin{equation}              \label{eq:idpsi}
 \Psi_{\gamma,v_0}(v_k) - \zeta_{v_{k-1}}^{\gamma}(v_k)\Psi_{\gamma,v_0}(v_{k-1}) = \pi^{-1}\Im \zeta_{v_{k-1}}^{\gamma}(v_k) \cdot \overline{G(v_0,v_{k-1};\gamma)} \, .
\end{equation}
\end{lem}
\begin{proof}
The first three assertions follow from (\ref{eq:green1}), (\ref{eq:green2}) and (\ref{eq:multigreen}), respectively.

By (\ref{eq:multigreen2}), we have $G(v_0,v_k;\gamma) = \zeta_{v_1}^{\gamma}(v_0) \frac{-\prod_{j=1}^{k-1} \zeta_{v_{j+1}}^{\gamma}(v_j)}{2m^{\gamma}_{v_k}} = \zeta_{v_1}^{\gamma}(v_0) G(v_1,v_k;\gamma)$. Next, on the path $(v_k,v_{k-1},\dots,v_0)$, we have $G(v_0,v_k;\gamma) = \overline{G(v_k,v_0; \text{\ovF{$\gamma$}})} = \overline{\zeta_{v_{k-1}}^{ \text{\ovF{$\gamma$}}}(v_k)G(v_{k-1},v_0;\text{\ovF{$\gamma$}})}  = \zeta_{v_{k-1}}^{\gamma}(v_k) G(v_0,v_{k-1};\gamma)$. In the last equality, we used $\overline{\zeta_w^{ \text{\ovF{$\gamma$}}}(v)} = -\overline{\langle \delta_v,(H^{(v|w)}- \text{\ovF{$\gamma$}})^{-1}\delta_v\rangle} = -\langle \delta_v, (H^{(v|w)}-\gamma)^{-1}\delta_v\rangle = \zeta_w^{\gamma}(v)$. This proves (\ref{eq:multigreen3}).

As in \cite[Equation (2.8)]{Klein}, one proves that $G^{(v_k|v_{k+1})}(v_0,v_k;\gamma) = -G^{(v_{k-1}|v_k)}(v_0,v_{k-1};\gamma) \cdot G^{(v_k|v_{k+1})}(v_k,v_k;\gamma)$ by studying $\mathcal{T}^{(v_k|v_{k+1})}$ instead of $\mathcal{T}$. The claim on $G^{(v_k|v_{k+1})}(v_0,v_k;\gamma)$ follows by induction. For $G^{(v_1|v_0)}(v_1,v_k;\gamma)$, consider $(v_k,\dots,v_1)$ as before.

Since $G(v,w;\gamma) = \zeta_w^{\gamma}(v) G(w,w;\gamma)$ and $G(v,w;\gamma) = \zeta_v^{\gamma}(w) G(v,v;\gamma)$ for $v\sim w$, we have $\zeta_w^{\gamma}(v) = \frac{G(v,v;\gamma)}{G(w,w;\gamma)}\,\zeta_v^{\gamma}(w) = \frac{m_w^{\gamma}}{m_v^{\gamma}}\,\zeta_v^{\gamma}(w)$. 

Next, by (\ref{eq:green3}), $\gamma = V(v)+ \sum_{u \sim v} \zeta_v^{\gamma}(u)+2m^{\gamma}_v = V(v)+ \sum_{u \in \mathcal{N}_v \setminus \{w\}} \zeta_v^{\gamma}(u) + \frac{1}{\zeta_w^{\gamma}(v)}$, so we get $2m^{\gamma}_v = \frac{1}{\zeta_w^{\gamma}(v)} - \zeta_v^{\gamma}(w)$. 

Next, let $(v_0,\dots,v_k)$ with $v_0=v$ and $v_k=w$. Then $G(v,w;\gamma) = G(v_0,v_k;\gamma) = \zeta_{v_1}^{\gamma}(v_0)G(v_1,v_k;\gamma) = \prod_{j=0}^{k-1} \zeta_{v_{j+1}}^{\gamma}(v_j)G(v_k,v_k;\gamma)$. Considering the path $(v_k,v_{k-1},\dots,v_0)$, we have $G(w,v;\gamma) = G(v_k,v_0;\gamma) = \zeta_{v_1}^{\gamma}(v_0) G(v_k,v_1;\gamma) =\prod_{j=0}^{k-1}\zeta_{v_{j+1}}^{\gamma}(v_k)G(v_k,v_k;\gamma)$. Thus, $G(v,w;\gamma)=G(w,v;\gamma)$.

By (\ref{eq:green3}), $\sum_{u\in \mathcal{N}_v\setminus \{w\}} \Im \zeta_v^{\gamma}(u)= \Im(\gamma - V(v) - \frac{1}{\zeta_w^{\gamma}(v)}) =  \frac{\Im \zeta_w^{\gamma}(v)}{|\zeta_w^{\gamma}(v)|^2} + \eta$. Since $\Im \zeta_w^{\gamma}(v)<0$ for any $\gamma\in \C^+$ and $v\sim w$, relation (\ref{eq:sumzeta}) follows.

Since $G(v_0,v_k;\gamma) = \zeta_{v_{k-1}}^{\gamma}(v_k)G(v_0,v_{k-1};\gamma)$, $\Im(zz') = (\Re z)(\Im z')+(\Im z)(\Re z')$ and $z(\Im z') = (\Re z)(\Im z') +i(\Im z)(\Im z')$, we have
\[
\Im G(v_0,v_k;\gamma) - \zeta_{v_{k-1}}^{\gamma}(v_k) \Im G(v_0,v_{k-1};\gamma) = \Im \zeta_{v_{k-1}}^{\gamma}(v_k) \cdot  \overline{G(v_0,v_{k-1};\gamma)} \, .
\]
so (\ref{eq:idpsi}) follows.
\end{proof}

\section{The Poisson kernel}                \label{sec:poisson}

An \emph{arc} is a non-backtracking path $(u_0,\ldots,u_k)$. If $v,w \in \mathcal{T}$, there is a unique arc joining $v$ to $w$; we denote it by $[v,w]$.

A \emph{chain} is an infinite non-backtracking path $(u_0,u_1,\dots)$. Two chains $(u_0,u_1,\ldots)$ and $(v_0,v_1,\ldots)$ are equivalent if $u_k = v_{k+n}$ for some $n \in \Z$ and all $k$. Any equivalence class of chains $\xi$ has a representative starting at an arbitrary $v\in \mathcal{T}$, which we denote by $[v,\xi]$.


The \emph{geometric boundary} $\partial \mathcal{T}$ of $\mathcal{T}$ is the set of equivalence classes of chains.

In the following, we fix a vertex $o \in \mathcal{T}$ and call it the origin. We denote $|v| := d(v,o)$.

Given $u\in \mathcal{T}$, we denote $\mathcal{N}_u^+ = \{w \sim u : |w|=|u|+1\}$.

Given $v,w \in \mathcal{T}$, $v \neq w$, we define
\[
\partial \mathcal{T}_{v,w} = \{ \xi \in \partial \mathcal{T} : [v,w] \text{ is a subchain of } [v,\xi] \} \quad \text{and} \quad \partial \mathcal{T}_w := \partial \mathcal{T}_{o,w} \, .
\]
Then for any $v\in\mathcal{T}$ and $n\in\mathbb{N}$, $\{ \partial \mathcal{T}_{v,w} : d(v,w) = n \}$ is a partition of $\partial \mathcal{T}$.

Given $v,\xi \in \mathcal{T} \cup \partial \mathcal{T}$, $v \neq \xi$, define $v \wedge \xi$ as the vertex with maximal length in $[o,v] \cap [o,\xi]$. We also set  $v \wedge v = v$ for $v\in \mathcal{T}$.

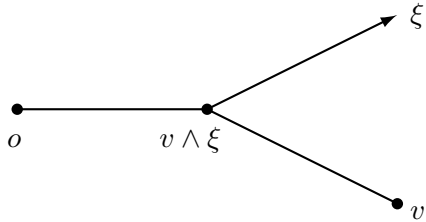
\setlength{\unitlength}{1mm}
\thicklines
\begin{picture}(40,40)
\put(70,20){
\vector(2,1){25}
}
\put(70,20){
\line(2,-1){25}
}
\put(45,20){
\line(1,0){25}
}
\put(45,20){
\circle*{1.5}
}
\put(70,20){
\circle*{1.5}
}
\put(95,7.5){
\circle*{1.5}
}
\put(45,15){$o$}
\put(65,15){$v\wedge \xi$}
\put(98,31.5){$\xi$}
\put(98,5.5){$v$}
\end{picture}

For a sequence $(v_n)$ of elements of $\mathcal{T} \cup \partial \mathcal{T}$, we say that
\begin{equation}\label{e:conv}
v_n\To \xi \qquad \text{if}\qquad |v_n \wedge \xi| \To +\infty \,.
\end{equation}
This notion does not depend on the choice of the origin $o$.

Let $\gamma \in \C^+$ and $\xi \in \partial \mathcal{T}$. We define the \emph{Poisson kernel} of $H$ by
\begin{equation}\label{e:poisson}
P_{\gamma,\xi}(v) := \frac{G(v\wedge \xi,v;\gamma)}{G(o,v\wedge \xi;\gamma)} \, .
\end{equation}
 
The following lemma collects its basic properties.

\begin{lem}               \label{lem:poissonexpan}
Fix $\xi \in \partial \mathcal{T}$ and $\gamma \in \C^+$.
\begin{enumerate}[\rm (a)]
\item If $(v_0,\dots,v_k)$ is a path with $v_0=o$, $v_k=v$ and $v_r=v\wedge \xi$, then
\[
P_{\gamma,\xi}(v) = \frac{\prod_{j=r}^{k-1} \zeta_{v_j}^{\gamma}(v_{j+1})}{\prod_{j=0}^{r-1}\zeta_{v_{j+1}}^{\gamma}(v_j)}  = \frac{G^{(v_{r+1}|v_r)}(v_{r+1},v_k;\gamma)}{G^{(v_{r-1}|v_r)}(v_0,v_{r-1};\gamma)} \, .
\]
\item Let $u\in \mathcal{T}$ and $u_+\in \mathcal{N}_u^+$. Then
\[
P_{\gamma,\xi}(u_+) = \begin{cases} \zeta_u^{\gamma}(u_+) P_{\gamma,\xi}(u)&\text{if }\xi \notin \partial \mathcal{T}_{u_+},\\ \frac{1}{\zeta_{u_+}^{\gamma}(u)}P_{\gamma,\xi}(u)&\text{if }\xi \in \partial \mathcal{T}_{u_+}. \end{cases}
\]
\item For any $v\in \mathcal{T}$, we have
\[
P_{\gamma,\xi}(v) = \lim_{u\to \xi} \frac{G(u,v;\gamma)}{G(o,u;\gamma)} \, .
\]
More precisely, if $v\in \mathcal{T}$ and $w=v \wedge \xi$, then for any $u \in [w,\xi]$, we have
\[
P_{\gamma,\xi}(v) = \frac{G(u,v;\gamma)}{G(o,u;\gamma)} \, .
\]
\item We have $H P_{\gamma,\xi} = \gamma P_{\gamma,\xi}$.
\end{enumerate}
\end{lem}
\begin{proof}
\begin{enumerate}[\rm (a)]
\item $P_{\gamma,\xi}(v) = \frac{G(v_r,v_k;\gamma)}{G(v_0,v_r;\gamma)} = \frac{\zeta_{v_{k-1}}^{\gamma}(v_k)G(v_r,v_{k-1};\gamma)}{\zeta_{v_1}^{\gamma}(v_0)G(v_1,v_r;\gamma)} = \frac{\prod_{j=r}^{k-1}\zeta_{v_j}^{\gamma}(v_{j+1})G(v_r,v_r;\gamma)}{\prod_{j=0}^{r-1}\zeta_{v_{j+1}}^{\gamma}(v_j) G(v_r,v_r;\gamma)} $, so the claim follows by (\ref{eq:multigreen4}).
\item If $\xi \notin \partial \mathcal{T}_{u_+}$, then $u_+$ is farther than $u$ to $\xi$, so $u_+ \wedge \xi = u \wedge \xi$ and $P_{\gamma,\xi}(u_+) = \frac{G(u_+\wedge\xi,u_+;\gamma)}{G(o,u_+\wedge\xi;\gamma)} = \frac{\zeta_u^{\gamma}(u_+)G(u\wedge\xi,u;\gamma)}{G(o,u\wedge\xi;\gamma)} = \zeta_u^{\gamma}(u_+)P_{\gamma,\xi}(u)$. If $\xi \in \partial \mathcal{T}_{u_+}$, then $u_+\wedge \xi = u_+$ and $u\wedge \xi = u$. Thus, $P_{\gamma,\xi}(u_+) = \frac{G(u_+,u_+;\gamma)}{G(o,u_+;\gamma)} = \frac{G(u,u;\gamma)}{\zeta_u^{\gamma}(u_+)G(o,u;\gamma)} \cdot \frac{\zeta_u^{\gamma}(u_+)}{\zeta_{u_+}^{\gamma}(u)} = \frac{1}{\zeta_{u_+}^{\gamma}(u)}P_{\gamma,\xi}(u)$.
\item Let $u \in [w,\xi]$ and let $(u_0,\dots,u_k)$ be an arc with $u_0=w$ and $u_k=u$. Then we have $G(o,u;\gamma) = G(o,u_k;\gamma) = \zeta_{u_{k-1}}^{\gamma}(u_k) G(o,u_{k-1};\gamma) = \prod_{j=0}^{k-1} \zeta_{u_j}^{\gamma}(u_{j+1}) G(o,u_0;\gamma) = \prod_{j=0}^{k-1} \zeta_{u_j}^{\gamma}(u_{j+1}) G(o,w;\gamma)$. On the inverted path $(v_0,\dots,v_k)$ with $v_0=u$ and $v_k=w$, we have $G(u,v;\gamma) = G(v_0,v;\gamma) = \zeta_{v_1}^{\gamma}(v_0)G(v_1,v;\gamma) = \prod_{j=0}^{k-1}\zeta_{v_{j+1}}^{\gamma}(v_j) G(v_k,v;\gamma) = \prod_{j=0}^{k-1} \zeta_{u_j}^{\gamma}(u_{j+1}) G(w,v;\gamma)$. Thus, $\frac{G(u,v;\gamma)}{G(o,u;\gamma)} = \frac{G(w,v;\gamma)}{G(o,w;\gamma)} = P_{\gamma,\xi}(v)$.
\item Let $f_v^{\gamma}(w) = G(v,w;\gamma)$. We first show that $H f_v^{\gamma} = \delta_v + \gamma f^{\gamma}_v$. Indeed, using (\ref{eq:greensym}),
\begin{align}\label{e:feigen}
(Hf_v^{\gamma})(w) & = \sum_u H(w,u)f_v^{\gamma}(u) = \sum_u H(w,u)G(v,u;\gamma) = \sum_u H(w,u) G(u,v;\gamma) \nonumber\\
& = [H(H-\gamma)^{-1}](w,v) = \langle \delta_w, H(H-\gamma)^{-1}\delta_v\rangle \nonumber \\
& = \langle \delta_w,\delta_v\rangle + \gamma \langle \delta_w, (H-\gamma)^{-1}\delta_v\rangle = \delta_v(w) + \gamma f^{\gamma}_v(w)
\end{align}
as asserted. Now let $v\in\mathcal{T}$ and $\xi\in\partial\mathcal{T}$, say $\xi=(o,s_1,s_2,\dots)$. Let $n>|v|+1$. Then by (c), for any $w\in\{v\}\cup \mathcal{N}_v$, we have $P_{\gamma,\xi}(w) = P_{\gamma,s_n}(w) = \frac{G(s_n,w;\gamma)}{G(o,s_n;\gamma)} = \frac{f_{s_n}^{\gamma}(w)}{f_o^{\gamma}(s_n)}$. Hence, $(HP_{\gamma,\xi})(v) = \sum_{w\sim v}\frac{f^{\gamma}_{s_n}(w)}{f^{\gamma}_o(s_n)} + V(v) \frac{f^{\gamma}_{s_n}(v)}{f_o^{\gamma}(s_n)} = \frac{(Hf^{\gamma}_{s_n})(v)}{f^{\gamma}_o(s_n)} = \frac{\delta_{s_n}(v)+\gamma f_{s_n}^{\gamma}(v)}{f_o^{\gamma}(s_n)} = \gamma P_{\gamma,\xi}(v)$.      \qedhere
\end{enumerate}
\end{proof}

\begin{rem}
One could define the Poisson kernel \eqref{e:poisson} alternatively as follows. Given $\gamma\in \C^+$, let $K_{\gamma,v}(u)=\frac{G(u,v;\gamma)}{G(o,u;\gamma)}$. Then $\mathscr{F}=(K_{\gamma,v})_{v\in\mathcal{T}}$ is a family of bounded functions on $\mathcal{T}$. Item (c) in Lemma~\ref{lem:poissonexpan} says that this family extends continuously to $\partial \mathcal{T}$ (in the sense \eqref{e:conv}) via the formula $K_{\gamma,v}(\xi)=K_{\gamma,v}(v\wedge \xi)$. One then defines $P_{\gamma,\xi}(v) := K_{\gamma,v}(\xi)$. Note that the family $\mathscr{F}$ separates the points of $\partial \mathcal{T}$~: if $\xi\neq \xi'\in \partial\mathcal{T}$, let $w=\xi \wedge \xi'$ and let $v\in \mathcal{N}_w \cap [w,\xi]$. Then $K_{\gamma,v}(\xi) = K_{\gamma,v}(v)$ and $K_{\gamma,v}(\xi') = K_{\gamma,v}(w)$. By \eqref{eq:multigreen3}, we have $G(w,v;\gamma)=\zeta_v^{\gamma}(w)G(v,v;\gamma)$ and $G(o,w;\gamma) = \frac{G(o,v;\gamma)}{\zeta_w^{\gamma}(v)}$, so $K_{\gamma,v}(w)=\frac{G(w,v;\gamma)}{G(o,w;\gamma)}=\zeta_v^{\gamma}(w)\zeta_w^{\gamma}(v)K_{\gamma,v}(v)$. Moreover, $\zeta_v^{\gamma}(w)\zeta_w^{\gamma}(v)\neq 1$ since $G(v,w;\gamma) = \frac{-\zeta_w^{\gamma}(v)}{2m_w^{\gamma}} = \frac{-\zeta_w^{\gamma}(v)}{\frac{1}{\zeta_v^{\gamma}(w)}-\zeta_w^{\gamma}(v)}=\frac{-\zeta_w^{\gamma}(v)\zeta_v^{\gamma}(w)}{1-\zeta_v^{\gamma}(w)\zeta_w^{\gamma}(v)}$ and $|G(v,w;\gamma)|\le \frac{1}{\Im \gamma}<\infty$. Hence, $K_{\gamma,v}(\xi)\neq K_{\gamma,v}(\xi')$. It follows that the geometric compactification $\mathcal{T}\cup \partial \mathcal{T}$ coincides with the compactification $\widehat{\mathcal{T}}_{\mathscr{F}}$ induced by $\mathscr{F}$, see e.g. \cite[Theorem 7.13]{Woess}. The previous argument is very similar to the one in \cite[Chapter 9.C]{Woess}, which shows that the Martin compactification of a transient nearest-neighbor random walk on $\mathcal{T}$ coincides with $\mathcal{T}\cup \partial \mathcal{T}$.
\end{rem}

Item (d) in Lemma~\ref{lem:poissonexpan} shows that the Poisson kernel is a ``generalized eigenfunction''. The following theorem shows that any generalized eigenfunction with eigenvalue $\gamma\in\IC^+$ can actually be expanded in Poisson kernels. Let $\mathcal{M}$ be the algebra generated by the sets $(\partial \mathcal{T}_v)_{v\in \mathcal{T}}$.

\begin{thm}                \label{thm:gef}
Let $\gamma \in \C^+$ and $f:\mathcal{T}\to \C$.
\begin{enumerate}[\rm (i)]
\item If $f(v) = \int_{\partial \mathcal{T}} P_{\gamma,\xi}(v)\,\dd \nu(\xi)$ for some finitely additive measure $\nu$ on $\mathcal{M}$, then $Hf = \gamma f$. Moreover, we must have
\begin{equation}         \label{eq:nudef}
\nu(\partial \mathcal{T}) = f(o) \quad \text{and} \quad \nu(\partial \mathcal{T}_{u_+}) = - G(o,u;\gamma)\left\{f(u_+) - \zeta_u^{\gamma}(u_+) f(u)\right\}
\end{equation}
for any $u\in \mathcal{T}$ and $u_+\in \mathcal{N}_u^+$.
\item Conversely, if $Hf=\gamma f$, the assignment \emph{(\ref{eq:nudef})} defines a finitely additive measure $\nu$ on $\mathcal{M}$ such that $f(v) = \int_{\partial \mathcal{T}} P_{\gamma,\xi}(v)\,\dd \nu(\xi)$.
\end{enumerate}
\end{thm}
\begin{proof}
\begin{enumerate}[\rm (i)]
\item Suppose $f(v) = \int_{\partial\mathcal{T}} P_{\gamma,\xi}(v)\,\dd \nu(\xi)$. Since for each $\xi$ we have $HP_{\gamma,\xi} = \gamma P_{\gamma,\xi}$, it follows that $Hf = \gamma f$. Indeed, let $v\in\mathcal{T}$, $\xi=(o,s_1,s_2,\dots)$ and $n>|v|+1$. Then $P_{\gamma,\xi}(w) = P_{\gamma,s_n}(w)$ for all $w\in \{v\}\cup \mathcal{N}_v$. Hence,
\begin{multline*}
\gamma f(v) = \int_{\partial \mathcal{T}} (HP_{\gamma,\xi})(v)\,\dd\nu(\xi) = \sum_{|s_n|=n} \Big[\Big(\sum_{w\sim v} P_{\gamma,s_n}(w)\Big) + V(v)P_{\gamma,s_n}(v)\Big]\nu(\partial \mathcal{T}_{s_n}) \\
= \sum_{w\sim v}\sum_{|s_n|=n} P_{\gamma,s_n}(w)\nu(\partial\mathcal{T}_{s_n}) + V(v)\sum_{|s_n|=n}P_{\gamma,s_n}(v)\nu(\partial\mathcal{T}_{s_n}) = (Hf)(v)
\end{multline*}
as asserted. Moreover, $f(o) = \int_{\partial \mathcal{T}} P_{\gamma,\xi}(o)\,\dd \nu(\xi) = \nu(\partial \mathcal{T})$ as claimed.

Given $u\in \mathcal{T}$ and $u_+\in \mathcal{N}_u^+$, we have $P_{\gamma,\xi}(u_+) = \zeta_u^{\gamma}(u_+) P_{\gamma,\xi}(u)$ if $\xi \notin \partial \mathcal{T}_{u_+}$, while $P_{\gamma,\xi}(u_+) = \frac{1}{\zeta_{u_+}^{\gamma}(u)} P_{\gamma,\xi}(u)$ if $\xi \in \partial \mathcal{T}_{u_+}$ by Lemma~\ref{lem:poissonexpan}. Hence,
\begin{align*}
f(u_+) & = \int_{\partial \mathcal{T}} P_{\gamma,\xi}(u_+)\,\dd \nu(\xi) = \int_{\partial \mathcal{T} \setminus \partial \mathcal{T}_{u_+}} P_{\gamma,\xi}(u_+)\,\dd \nu(\xi) + \int_{\partial \mathcal{T}_{u_+}} P_{\gamma,\xi}(u_+)\,\dd \nu(\xi) \\
& = \zeta_u^{\gamma}(u_+)\int_{\partial \mathcal{T} \setminus \partial \mathcal{T}_{u_+}} P_{\gamma,\xi}(u)\,\dd \nu(\xi) + \frac{1}{\zeta_{u_+}^{\gamma}(u)} \int_{\partial \mathcal{T}_{u_+}} P_{\gamma,\xi}(u)\,\dd \nu(\xi) \\
& = \zeta_u^{\gamma}(u_+)\int_{\partial \mathcal{T}} P_{\gamma,\xi}(u)\,\dd \nu(\xi) + \Big(\frac{1}{\zeta_{u_+}^{\gamma}(u)} - \zeta_u^{\gamma}(u_+)\Big)\int_{\partial \mathcal{T}_{u_+}} P_{\gamma,\xi}(u)\,\dd \nu(\xi)
\end{align*}
By assumption, $\int_{\partial \mathcal{T}} P_{\gamma,\xi}(u)\,\dd \nu(\xi)=f(u)$. Also, if $\xi \in \partial \mathcal{T}_{u_+}$, then $u\wedge \xi = u$, so $P_{\gamma,\xi}(u) = \frac{G(u,u;\gamma)}{G(o,u;\gamma)}$. Using (\ref{eq:mv}) we thus get
\begin{align*}
f(u_+) & =  \zeta_u^{\gamma}(u_+) f(u) + 2m_u^{\gamma} \,\frac{G(u,u;\gamma)}{G(o,u;\gamma)} \nu(\partial \mathcal{T}_{u_+}) \, ,
\end{align*}
so $ \nu(\partial \mathcal{T}_{u_+}) = - G(o,u;\gamma)\left\{f(u_+) - \zeta_u^{\gamma}(u_+) f(u)\right\}$ as asserted.
\item Suppose $Hf=\gamma f$. To see that $\nu$ is finitely additive, it suffices to show that for any $u\in \mathcal{T}$, we have $\nu(\partial \mathcal{T}_u) = \sum_{u_+\in \mathcal{N}_u^+} \nu(\partial \mathcal{T}_{u_+})$. For this, given $u\neq o$, let $u_-$ be the unique neighbor of $u$ with $|u_-|=|u|-1$. Then using (\ref{eq:green3}) and (\ref{eq:multigreen3}), we have
\begin{align*}
& \sum_{u_+\in \mathcal{N}_u^+} \nu(\partial \mathcal{T}_{u_+}) = -G(o,u;\gamma)\Big(\sum_{u_+\in \mathcal{N}_u^+} f(u_+) - f(u)\sum_{u_+\in \mathcal{N}_u^+} \zeta_u^{\gamma}(u_+)\Big) \\
& \qquad = -G(o,u;\gamma)\Big([(Hf)(u) - f(u_-) - V(u)f(u)] - f(u)\big[\gamma-V(u)-\frac{1}{\zeta_{u_-}^{\gamma}(u)}\big]\Big) \\
& \qquad = -\frac{G(o,u;\gamma)}{\zeta_{u_-}^{\gamma}(u)} \big( f(u) - \zeta_{u_-}^{\gamma}(u)f(u_-)\big) = \nu(\partial \mathcal{T}_u) \, .
\end{align*}
The case $u=o$ is similar. This proves finite additivity.

We next prove that $f(v) = \int_{\partial \mathcal{T}} P_{\gamma,\xi}(v)\,\dd \nu(\xi)$ using induction on $|v|$.

For $v=o$, we have $\int_{\partial \mathcal{T}} P_{\gamma,\xi}(o)\,\dd \nu(\xi) = \nu(\partial \mathcal{T}) = f(o)$ as asserted.

Suppose the relation is true for all vertices $u$ with $|u|=n$. Let $|u_+|=n+1$, say $u_+\in \mathcal{N}_u^+$ for some $u$ with $|u|=n$. Then using Lemma~\ref{lem:poissonexpan}, we have
\begin{align*}
& \int_{\partial \mathcal{T}} P_{\gamma,\xi}(u_+)\,\dd \nu(\xi) = \int_{\partial \mathcal{T} \setminus \partial \mathcal{T}_{u_+}} P_{\gamma,\xi}(u_+)\,\dd \nu(\xi) + \int_{\partial \mathcal{T}_{u_+}} P_{\gamma,\xi}(u_+)\,\dd \nu(\xi) \\
& \qquad = \zeta_u^{\gamma}(u_+) \int_{\partial \mathcal{T} \setminus \partial \mathcal{T}_{u_+}} P_{\gamma,\xi}(u)\,\dd \nu(\xi) + \frac{1}{\zeta_u^{\gamma}(u_+)} \int_{\partial \mathcal{T}_{u_+}} P_{\gamma,\xi}(u)\,\dd \nu(\xi) \\
& \qquad = \zeta_u^{\gamma}(u_+)\int_{\partial \mathcal{T}} P_{\gamma,\xi}(u)\,\dd \nu(\xi) +\Big(\frac{1}{\zeta_{u_+}^{\gamma}(u)} - \zeta_u^{\gamma}(u_+)\Big)\int_{\partial \mathcal{T}_{u_+}} P_{\gamma,\xi}(u)\,\dd \nu(\xi) \, .
\end{align*}
By the induction hypothesis, $\int_{\partial \mathcal{T}} P_{\gamma,\xi}(u)\,\dd \nu(\xi) = f(u)$. Also, if $\xi \in \partial \mathcal{T}_{u_+}$, then $u\wedge \xi=u$, so $P_{\gamma,\xi}(u) = \frac{G(u,u;\gamma)}{G(o,u;\gamma)}$. Using (\ref{eq:mv}) and (\ref{eq:nudef}), we thus get
\[
\int_{\partial \mathcal{T}} P_{\gamma,\xi}(u_+)\,\dd \nu(\xi) = \zeta_u^{\gamma}(u_+)f(u) + 2m_u^{\gamma}  \frac{G(u,u;\gamma)}{G(o,u;\gamma)} \nu(\partial \mathcal{T}_{u_+}) = f(u_+) \, .
\]
This completes the proof of (ii). \qedhere
\end{enumerate}
\end{proof}

Our target now is to extend the previous results to $\gamma=E+i0$. We start with the following lemma.

\begin{lem}              \phantomsection             \label{lem:limits}
There is a Lebesgue-null set $\mathfrak{A} \subset \R$ such that for any $E \in \mathfrak{S} := \R \setminus \mathfrak{A}$ and any $v \in \mathcal{T}$, $w\sim v$, the limits
\[
G(v,v;E +i0) := \lim_{\eta \downarrow 0} G(v,v;E+i\eta) \quad \text{and} \quad \zeta_w^{E+i0}(v):= \lim_{\eta \downarrow 0} \zeta_w^{E+i\eta}(v)
\]
exist, are finite and are non-zero.
\end{lem}
\begin{proof}
Let $\mu_v(J) = \langle \delta_v,\chi_J(H)\delta_v\rangle$ and $\mu_v^{(v|w)}(J) = \langle \delta_v,\chi_J(H^{(v|w)})\delta_v\rangle$ for Borel $J \subseteq \R$. Since $G(v,v;\gamma)$ and $\zeta_w^{\gamma}(v)$ are the Borel transforms of $\mu_v$ and $\mu^{(v|w)}_v$, respectively, we know the limits exist and are finite; see e.g. \cite[Theorem 1.4]{CRM}. Let $A^v$ and $A^{(v|w)}$ be the Lebesgue-null sets outside which the limits are finite. Put $\mathscr{A}= \cup_{v \in \mathcal{T}} A^v$, $\mathscr{A}' = \cup_{v \in \mathcal{T}} \cup_{w \sim v} \mathscr{A}^{(v|w)}$ and $\mathfrak{A} = \mathscr{A} \cup \mathscr{A}'$. Then $\mathfrak{A}$ is Lebesgue-null.

Let $E \notin \mathfrak{A}$. Since $G(v,v;\gamma) = \frac{1}{V(v) - \gamma + \sum_{u \sim v} \zeta_v^{\gamma}(u)}$, and since the limit of the denominator is finite as $\eta \downarrow 0$, we get $G(v,v;E+i0) \neq 0$. Similarly, we deduce from the identity $\zeta_w^{\gamma}(v) = \frac{1}{V(v) - \gamma + \sum_{u \in \mathcal{N}_v \setminus \{w\}} \zeta_v^{\gamma}(u)}$ that $\zeta_w^{E+i0}(v) \neq 0$.
\end{proof}
This directly implies the following proposition~: 
\begin{prp}\label{p:exists}
Let $E\in \mathfrak{S}$. Then for any $v\in \mathcal{T}$ and $\xi \in \partial \mathcal{T}$, the limit
\[
P_{E,\xi}(v) := \lim_{\eta \downarrow 0} P_{E+i\eta,\xi}(v)
\]
exists.
\end{prp}
\begin{proof}
By Lemma~\ref{lem:poissonexpan}, $P_{\gamma,\xi}(v) = \frac{\prod_{j=r}^{k-1} \zeta_{v_j}^{\gamma}(v_{j+1})}{\prod_{j=0}^{r-1}\zeta_{v_{j+1}}^{\gamma}(v_j)} $, so the claim follows from  Lemma~\ref{lem:limits}.
\end{proof}

Let $v,w\in\mathcal{T}$ and $\gamma\in\C^+$. Recall the notation
\[
\Psi_{\gamma,v}(w) = \frac{1}{\pi} \Im G(v,w;\gamma)
\]
introduced in Lemma~\ref{lem:zetapot}. The following lemma shows that in the regions where the operator $H$ has AC spectrum, we may expand the kernel of $H$ in terms of the explicit generalized eigenfunctions $\Psi_{E,v}$. In other words, \eqref{eq:decompo2} holds with $Q_{E,w}(v)=\Psi_{E,v}(w)$ and $\dd \rho_H(E) =\dd E$. This will later be combined with the previous lemmas on the Poisson kernel to prove the main result.

\begin{lem}\label{lem:1stexp}
Denote $\Psi_{E,v} := \lim_{\eta \downarrow 0} \Psi_{E+i\eta,v}$ when the limit exists.
\begin{enumerate}[\rm (i)]
\item If $H$ has purely absolutely continuous spectrum in $I\subset \R$, then for any bounded Borel $F:I\to \C$, and for any $v,w\in\mathcal{T}$, we have $F(H)(v,w) = \int_I F(E) \Psi_{E,v}(w)\,\dd E$.
\item For any $E\in \mathfrak{S}$, we have $H \Psi_{E,v} = E \Psi_{E,v}$.
\end{enumerate}
\end{lem}
\begin{proof}
\begin{enumerate}[\rm (i)]
\item Denote $\mu_{\phi,\psi}(J) = \langle \phi,\chi_J(H)\psi\rangle$ for Borel $J\subseteq \R$, $\mu_{\phi}=\mu_{\phi,\phi}$ and $\mu_{v,w}=\mu_{\delta_v,\delta_w}$. Since the spectrum is purely absolutely continuous in $I$, all measures $\mu_{\phi,\psi}$ are absolutely continuous in w.r.t. the Lebesgue measure in $I$.

By the spectral theorem $F(H)(v,w) = \int_I F(E)\,\dd \mu_{v,w}(E)$. We now show that $\dd \mu_{v,w}(E) = \Psi_{E,v}(w)\,\dd E$ in $I$.

Since $\mu_{\phi}$ is a finite positive measure which is absolutely continuous on $I$, we have by \cite[Theorem 1.6]{CRM} that $\mu_{\phi}(J) = \pi^{-1}\int_J \Im \langle \phi, (H-E-i0)^{-1}\phi\rangle\,\dd E$ for any $J\subseteq I$. Now note that
\begin{equation}\label{:e:nalini}
\langle \delta_v+\delta_w,A(\delta_v+\delta_w)\rangle - \langle \delta_v-\delta_w,A(\delta_v-\delta_w)\rangle = 2 \langle \delta_v,A\delta_w\rangle + 2\langle \delta_w,A\delta_v\rangle\,.
\end{equation}
Taking $A=(H-\gamma)^{-1}$ and using \eqref{eq:greensym}, we get
\begin{equation}\label{e:formulaforgreen}
G(v,w;\gamma) = \frac{\langle \phi,(H-\gamma)^{-1}\phi\rangle - \langle \psi,(H-\gamma)^{-1}\psi\rangle}{4}
\end{equation}
for $\phi=\delta_v+\delta_w$ and $\psi=\delta_v-\delta_w$. If $\gamma=E+i\eta$, taking $\eta\downarrow 0$ on a Lebesgue full set, we get for $J\subseteq I$,
\begin{align*}
\int_J \Psi_{E,v}(w)\,\dd E  & = \frac{1}{4\pi} \int_J \Im \langle \phi,(H-E-i0)^{-1}\phi\rangle\,\dd E \\
& \quad - \frac{1}{4\pi}\int_J \Im \langle \psi,(H-E-i0)^{-1}\psi\rangle\,\dd E \,.
\end{align*}
Applying \eqref{:e:nalini} with $A=\chi_J(H)$, we thus get
\[
\int_J \Psi_{E,v}(w)\,\dd E  = \frac{\mu_{\phi}(J) - \mu_{\psi}(J)}{4} = \frac{\mu_{v,w}(J) + \mu_{w,v}(J)}{2} = \mu_{v,w}(J) \,.
\]
as asserted. Here we used that $\mu_{v,w}=\mu_{w,v}$. This follows e.g. from \eqref{eq:greensym} using the relation $\frac{\mu_{v,w}[a,b]+\mu_{v,w}(a,b)}{2}=\lim_{\eta \downarrow 0} \frac{1}{\pi}\int_a^b \Im G(v,w;E+i\eta)\,\dd E$, which is a consequence of Fubini's theorem (regardless of the continuity of $\mu_{v,w}$).

\item We showed in \eqref{e:feigen} that if $f_v^{\gamma}(w) = G(v,w;\gamma)$, then $H f_v^{\gamma} = \delta_v + \gamma f_v^{\gamma}$. If $\gamma=E+i\eta$ and $\Psi_{\gamma,v}(w) = \frac{1}{\pi}\Im G(v,w;\gamma)$, we thus have $H \Psi_{\gamma,v} = \frac{\eta}{\pi} \Re f_v^{\gamma} + E \Psi_{\gamma,v}$. Assume $E\in \mathfrak{S}$. Then taking $\eta \downarrow 0$, we get using \cite[Theorem 1.6]{CRM} along with \eqref{e:formulaforgreen} that $\eta \Re f_v^{\gamma} \to 0$. Hence, $H\Psi_{E,v} = E\Psi_{E,v}$.\qedhere
\end{enumerate}
\end{proof}

Theorem~\ref{thm:gef} clearly continues to hold if we replace $\gamma\in \C^+$ by $\gamma=E+i0$, $E\in \mathfrak{S}$. So we may apply Theorem~\ref{thm:gef} to the generalized eigenfunction $\Psi_{E,v}$, assuming $E\in \mathfrak{S}$. Our next aim is to refine this expansion.

\begin{lem}                 \label{lem:psiv}
For any $E\in \mathfrak{S}$, we have
\[
\Psi_{E,v}(w) = \int_{\partial \mathcal{T}} \overline{P_{E,\xi}(v)} P_{E,\xi}(w)\,\dd \nu_E(\xi) \, ,
\]
where $\nu_E(\partial \mathcal{T}) = \Psi_{E,o}(o)$, and if $u_+\in \mathcal{N}_u^+$, then
\[
\nu_E(\partial \mathcal{T}_{u_+}) = \frac{1}{\pi} \cdot |G(o,u;E+i0)|^2\cdot |\Im \zeta_u^{E+i0}(u_+)| \, .
\]
\end{lem}
Note that since $\nu_E$ is non-negative, it extends to a countably additive measure on the $\sigma$-algebra generated by the sets $\{\partial \mathcal{T}_v\}$; see e.g. \cite{CCS}.
\begin{proof}
We first assume $v=o$. We know by Theorem~\ref{thm:gef} (ii) that
\begin{equation}          \label{eq:psio}
\Psi_{E,o}(w) = \int_{\partial \mathcal{T}} P_{E,\xi}(w)\,\dd \nu_E(\xi) \, ,
\end{equation}
with $\nu_E(\partial \mathcal{T})=\Psi_{E,o}(o)$ and $\nu_E(\partial \mathcal{T}_{u_+}) = -G(o,u;E+i0) \{ \Psi_{E,o}(u_+) - \zeta_u^{E+i0}(u_+) \Psi_{E,o}(u) \}$. If $(v_0,\dots,v_k)$ is an arc with $v_0=o$ and $v_k=u_+$ (so that $v_{k-1}=u$), then using (\ref{eq:idpsi}), we get $\nu_E(\partial \mathcal{T}_{u_+}) = \frac{-1}{\pi}G(o,u;E+i0) \cdot \Im \zeta_u^{E+i0}(u_+) \cdot \overline{G(o,u;E+i0)}$.

This proves the claim for $v=o$, since $P_{E,\xi}(o)=1$. Now let $v\in \mathcal{T}$. Since $o\in \mathcal{T}$ is arbitrary, by placing the origin at $v$, we get by (\ref{eq:psio}),
\begin{equation}               \label{eq:poissonprel}
\Psi_{E,v}(w) = \int_{\partial \mathcal{T}} P_{E,\xi}^{(v)}(w)\,\dd \nu_{E,v}(\xi) \, .
\end{equation}
Here
\[
P_{E,\xi}^{(v)}(w) = \frac{G(t_0,w;E+i0)}{G(v,t_0;E+i0)} \, ,
\]
where $t_0$ is the vertex of maximal distance from $v$ in $[v,w] \cap [v,\xi]$ and if $u\in \mathcal{T}$, and $u_+\sim u$ has $d(u_+,v)=d(u,v)+1$, then
\[
\nu_{E,v}(\partial \mathcal{T}_{v,u_+}) = \frac{1}{\pi}\cdot |G(v,u;E+i0)|^2 \cdot |\Im \zeta_u^{E+i0}(u_+)| \, .
\]
One can show as before that $P_{E,\xi}^{(v)}(w) = \lim_{t \to \xi} \frac{G(t,w;E+i0)}{G(v,t;E+i0)}$. Using (\ref{eq:greensym}), we have
\[
\frac{G(t,w;E+i0)}{G(v,t;E+i0)} = \frac{G(t,w;E+i0)}{G(o,t;E+i0)} \cdot \frac{G(o,t;E+i0)}{G(t,v;E+i0)} \, .
\]
Taking the limit as $t\to \xi$, we thus get
\begin{equation}            \label{eq:pexiv}
P_{E,\xi}^{(v)}(w)  = \frac{P_{E,\xi}(w)}{P_{E,\xi}(v)} \, .
\end{equation}
Now fix $\xi = (o,s_1,s_2,\dots) \in \partial \mathcal{T}$ and let $v\in \mathcal{T}$. Say $|v|<n$ for some $n$. We will show that
\begin{equation}            \label{eq:raniko}
\nu_{E,v}(\partial \mathcal{T}_{s_n}) = |P_{E,\xi}(v)|^2\cdot \nu_E(\partial \mathcal{T}_{s_n}) \, .
\end{equation}
First note that since $|v|<|s_n|$, then $\partial \mathcal{T}_{s_n} = \partial \mathcal{T}_{v,s_n}$. Indeed, any $\xi' \in \partial \mathcal{T}_{s_n}$ is equivalent to the element of $\partial \mathcal{T}_{v,s_n}$ sharing the infinite intersection $[s_n,\xi']$ and vice versa. Hence,
\begin{equation}            \label{eq:nuvmesure}
\nu_{E,v}(\partial \mathcal{T}_{s_n}) = \frac{1}{\pi}\cdot |G(v,s_{n-1};E+i0)|^2\cdot |\Im \zeta_{s_{n-1}}^{E+i0}(s_n)| \, ,
\end{equation}
because $|v|<n$ implies $d(v,s_n) = d(v,s_{n-1})+1$. Let $s_r = v\wedge \xi$ and let $(v_0,\dots,v_k)$ be an arc with $v_0=o$ and $v_k=v$. By definition, $v_j=s_j$ for all $j \le r$. Now, considering the arc $(v_k,v_{k-1},\dots,v_r,s_{r+1},\dots,s_{n-1})$, we have $G(v,s_{n-1};\gamma) = G(v_k,s_{n-1};\gamma) = \zeta_{v_{k-1}}^{\gamma}(v_k) G(v_{k-1},s_{n-1};\gamma) = \prod_{j=r}^{k-1} \zeta_{v_j}^{\gamma}(v_{j+1}) G(v_r,s_{n-1};\gamma)$. Furthermore, $G(o,s_{n-1};\gamma) = \zeta_{s_1}^{\gamma}(o)G(s_1,s_{n-1};\gamma) = \prod_{j=0}^{r-1} \zeta_{v_{j+1}}^{\gamma}(v_j) G(v_r,s_{n-1};\gamma)$ because $v_j=s_j$ for all $j \le r$. Hence,
\[
G(v,s_{n-1};\gamma) = \prod_{j=r}^{k-1} \zeta_{v_j}^{\gamma}(v_{j+1}) G(v_r,s_{n-1};\gamma) = \frac{\prod_{j=r}^{k-1} \zeta_{v_j}^{\gamma}(v_{j+1})}{\prod_{j=0}^{r-1} \zeta_{v_{j+1}}^{\gamma}(v_j) } G(o,s_{n-1};\gamma) \, .
\]
It follows by (\ref{eq:nuvmesure}) and Lemma~\ref{lem:poissonexpan} that
\[
\nu_{E,v}(\partial \mathcal{T}_{s_n}) = \frac{1}{\pi} \cdot |P_{E,\xi}(v)|^2 \cdot |G(o,s_{n-1};E+i0)|^2 \cdot |\Im \zeta_{s_{n-1}}^{E+i0}(s_n)|  \, ,
\]
which proves (\ref{eq:raniko}). Finally, choosing $n$ such that $n>\max(|v|,|w|)$, we have $P_{E,\xi'}(v) = P_{E,s_n}(v)$ and $P_{E,\xi'}(w) = P_{E,s_n}(w)$ for any $\xi' \in \partial \mathcal{T}_{s_n}$. So by (\ref{eq:poissonprel}) and (\ref{eq:pexiv}),
\begin{align*}
\Psi_{E,v}(w) & = \sum_{|s_n|=n} \frac{P_{E,s_n}(w)}{P_{E,s_n}(v)}\,\nu_{E,v}(\partial \mathcal{T}_{s_n}) = \sum_{|s_n|=n}\frac{P_{E,s_n}(w)}{P_{E,s_n}(v)}\cdot |P_{E,s_n}(v)|^2\,\nu_E(\partial \mathcal{T}_{s_n}) \\
& = \int_{\partial \mathcal{T}} \overline{P_{E,\xi'}(v)}P_{E,\xi'}(w)\,\dd \nu_E(\xi') \, .            \qedhere
\end{align*}
\end{proof}

We may finally prove our main result.

\begin{thm}[Fourier transform, Plancherel formula]                  \label{thm:planch}
Suppose $H$ has purely absolutely continuous spectrum in some measurable set $I$ and let $\nu_E$ be the measure constructed in Lemma~\ref{lem:psiv}.
\begin{enumerate}[\rm (i)]
\item For any bounded Borel $F:I\to \C$ and any $v,w\in \mathcal{T}$, we have
\[
F(H)(v,w) = \int_I F(E) \Psi_{E,v}(w)\,\dd E = \int_I \int_{\partial \mathcal{T}} F(E) P_{E,\xi}(v)\overline{P_{E,\xi}(w)}\,\dd \nu_E(\xi) \,\dd E \, .
\]
\item For any $f$ on $\mathcal{T}$ with finite support and any bounded Borel $F:I\to \C$, we have
\[
[F(H)f](v) = \int_I F(E) \langle \text{\ovF{$f$}},\Psi_{E,v}\rangle\,\dd E = \int_I \int_{\partial \mathcal{T}} F(E) P_{E,\xi}(v) \langle P_{E,\xi}, f\rangle\,\dd \nu_E(\xi)\,\dd E \, .
\]
\item For any $K$ on $\mathcal{T} \times \mathcal{T}$ with finite support and any bounded Borel $F:I \to \C$, we have
\[
\tr[F(H)K] = \int_I \int_{\partial \mathcal{T}} F(E) \langle P_{E,\xi},KP_{E,\xi}\rangle\,\dd \nu_E(\xi)\,\dd E \, .
\]
\end{enumerate}
\end{thm}
\begin{proof}
We proved that $F(H)(v,w) = \int_I F(E) \Psi_{E,v}(w)\,\dd E$ in Lemma~\ref{lem:1stexp}. Since $\Psi_{E,v}(w) = \Psi_{E,w}(v)$ by (\ref{eq:greensym}), we obtain (i) using Lemma~\ref{lem:psiv} and the fact that $\int_I = \int_{I \,\cap \,\mathfrak{S}}$.

Next, given $f$ with finite support, say $f=\sum_w f(w)\delta_w$, we have
\[
[F(H)f](v) = \langle \delta_v,F(H)f\rangle = \sum_{w\in \mathcal{T}} f(w)\langle \delta_v, F(H)\delta_w\rangle = \sum_{w\in \mathcal{T}} f(w)\int F(E) \Psi_{E,v}(w)\,\dd E \, .
\]
On one hand this equals $\int F(E)  \langle \text{\ovF{$f$}},\Psi_{E,v}\rangle\,\dd E$, on the other hand, if we use (i), we see it is equal to $\int_I \int_{\partial \mathcal{T}} F(E) P_{E,\xi}(v) \langle P_{E,\xi}, f\rangle\,\dd \nu_E(\xi)\,\dd E $.

Finally, given $K$ with finite support, we have by (i),
\begin{align*}
& \int_{I\times \partial \mathcal{T}} F(E)\langle P_{E,\xi}, K P_{E,\xi} \rangle \, \dd \nu_E(\xi) \dd E   = \int_{I \times \partial \mathcal{T}} \sum_{w \in \mathcal{T}} F(E)\overline{P_{E,\xi}(w)} (KP_{E,\xi})(w) \,\dd \nu_E(\xi) \dd E \\
& \qquad = \sum_{v,w \in \mathcal{T}} K(w,v) \int_{I \times \partial \mathcal{T}} F(E)\overline{P_{E,\xi}(w)} P_{E,\xi}(v) \,\dd \nu_E(\xi) \dd E \\
& \qquad = \sum_{v,w\in \mathcal{T}} K(w,v) F(H)(v,w) = \tr[K F(H)] = \tr[F(H)K] \, .            \qedhere
\end{align*}
\end{proof}

\begin{rem}\label{r:more}
The results of this paper generalize without difficulty to self-adjoint operators of the form $(H_p\psi)(v) = \sum_w p_v(w) f(w)$, where $p_v(w) = 0$ if $d(v,w)>1$, assuming all coefficients $p_v(w)$ are real, with $p_v(w)=p_w(v)$ and $p_v(w)\neq 0$ whenever $v\sim w$. In this case, relation (\ref{eq:green1}) becomes $G_p(v,v;\gamma) = \frac{1}{p_v(v) - \gamma - \sum_{u\sim v} p_v(u)p_u(v)G_p^{(u|v)}(u,u;\gamma)}$, while (\ref{eq:multigreen}) becomes $G_p(v_0,v_k;\gamma) = (-1)^k\prod_{j=0}^{k-1}p_{v_j}(v_{j+1})G_p^{(v_j|v_{j+1})}(v_j,v_j;\gamma)\cdot G_p(v_k,v_k;\gamma)$. We then put $G_p(v,v;\gamma) = \frac{-1}{2m_v^{\gamma}}$ and $\zeta_w^{\gamma}(v) = -p_v(w) G_p^{(v|w)}(v,v;\gamma)$. Then one may obtain similar expansions in Poisson kernels, with minor modifications in the formulas of $\nu_E$. Note that $p_v(v)$ plays the role of $V(v)$ for such operators.
 \end{rem}
 
 \medskip
 
 {\bf{Acknowledgements~:}} This material is based upon work supported by the Agence Nationale de la Recherche under grant No.ANR-13-BS01-0007-01, by the Labex IRMIA and the Institute of Advance Study of Universit\'e de Strasbourg, and by Institut Universitaire de France.

\providecommand{\bysame}{\leavevmode\hbox to3em{\hrulefill}\thinspace}
\providecommand{\MR}{\relax\ifhmode\unskip\space\fi MR }
\providecommand{\MRhref}[2]{%
  \href{http://www.ams.org/mathscinet-getitem?mr=#1}{#2}
}
\providecommand{\href}[2]{#2}

\end{document}

%% file: def.tex
\newcommand{\nwc}{\newcommand}
\nwc{\nwt}{\newtheorem}
\nwt{coro}{Corollary}
\nwt{ex}{Example}
\nwt{prop}{Proposition}
\nwt{defin}{Definition}

\nwc{\mf}{\mathbf} 
\nwc{\blds}{\boldsymbol} 
\nwc{\ml}{\mathcal} 

\nwc{\lam}{\lambda}
\nwc{\del}{\delta}
\nwc{\Del}{\Delta}
\nwc{\Lam}{\Lambda}
\nwc{\elll}{\ell}

\nwc{\IA}{\mathbb{A}} 
\nwc{\IB}{\mathbb{B}} 
\nwc{\IC}{\mathbb{C}} 
\nwc{\ID}{\mathbb{D}} 
\nwc{\IE}{\mathbb{E}} 
\nwc{\IF}{\mathbb{F}} 
\nwc{\IG}{\mathbb{G}} 
\nwc{\IH}{\mathbb{H}} 
\nwc{\IN}{\mathbb{N}} 
\nwc{\IP}{\mathbb{P}} 
\nwc{\IQ}{\mathbb{Q}} 
\nwc{\IR}{\mathbb{R}} 
\nwc{\IS}{\mathbb{S}} 
\nwc{\IT}{\mathbb{T}} 
\nwc{\IZ}{\mathbb{Z}} 

\def\bbleft{{\mathchoice {[\mskip-3mu {[}} {[\mskip-3mu {[}}{[\mskip-4mu {[}}{[\mskip-5mu {[}}}}
\def\bbright{{\mathchoice {]\mskip-3mu {]}} {]\mskip-3mu {]}}{]\mskip-4mu {]}}{]\mskip-5mu {]}}}}
\nwc{\setK}{\bbleft 1,K \bbright}
\nwc{\setN}{\bbleft 1,\cN \bbright}
 
\nwc{\va}{{\bf a}}
\nwc{\vb}{{\bf b}}
\nwc{\vc}{{\bf c}}
\nwc{\vd}{{\bf d}}
\nwc{\ve}{{\bf e}}
\nwc{\vf}{{\bf f}}
\nwc{\vg}{{\bf g}}
\nwc{\vh}{{\bf h}}
\nwc{\vi}{{\bf i}}
\nwc{\vI}{{\bf I}}
\nwc{\vj}{{\bf j}}
\nwc{\vk}{{\bf k}}
\nwc{\vl}{{\bf l}}
\nwc{\vm}{{\bf m}}
\nwc{\vM}{{\bf M}}
\nwc{\vn}{{\bf n}}
\nwc{\vo}{{\it o}}
\nwc{\vp}{{\bf p}}
\nwc{\vq}{{\bf q}}
\nwc{\vr}{{\bf r}}
\nwc{\vs}{{\bf s}}
\nwc{\vt}{{\bf t}}
\nwc{\vu}{{\bf u}}
\nwc{\vv}{{\bf v}}
\nwc{\vw}{{\bf w}}
\nwc{\vx}{{\bf x}}
\nwc{\vy}{{\bf y}}
\nwc{\vz}{{\bf z}}
\nwc{\bal}{\blds{\alpha}}
\nwc{\bep}{\blds{\epsilon}}
\nwc{\barbep}{\overline{\blds{\epsilon}}}
\nwc{\bnu}{\blds{\nu}}
\nwc{\bmu}{\blds{\mu}}
\nwc{\bet}{\blds{\eta}}

\nwc{\bk}{\blds{k}}
\nwc{\bm}{\blds{m}}
\nwc{\bM}{\blds{M}}
\nwc{\bp}{\blds{p}}
\nwc{\bq}{\blds{q}}
\nwc{\bn}{\blds{n}}
\nwc{\bv}{\blds{v}}
\nwc{\bw}{\blds{w}}
\nwc{\bx}{\blds{x}}
\nwc{\bxi}{\blds{\xi}}
\nwc{\by}{\blds{y}}
\nwc{\bz}{\blds{z}}

\nwc{\cA}{\ml{A}}
\nwc{\cB}{\ml{B}}
\nwc{\cC}{\ml{C}}
\nwc{\cD}{\ml{D}}
\nwc{\cE}{\ml{E}}
\nwc{\cF}{\ml{F}}
\nwc{\cG}{\ml{G}}
\nwc{\cH}{\ml{H}}
\nwc{\cI}{\ml{I}}
\nwc{\cJ}{\ml{J}}
\nwc{\cK}{\ml{K}}
\nwc{\cL}{\ml{L}}
\nwc{\cM}{\ml{M}}
\nwc{\cN}{\ml{N}}
\nwc{\cO}{\ml{O}}
\nwc{\cP}{\ml{P}}
\nwc{\cQ}{\ml{Q}}
\nwc{\cR}{\ml{R}}
\nwc{\cS}{\ml{S}}
\nwc{\cT}{\ml{T}}
\nwc{\cU}{\ml{U}}
\nwc{\cV}{\ml{V}}
\nwc{\cW}{\ml{W}}
\nwc{\cX}{\ml{X}}
\nwc{\cY}{\ml{Y}}
\nwc{\cZ}{\ml{Z}}

\nwc{\fA}{\mathfrak{a}}
\nwc{\fB}{\mathfrak{b}}
\nwc{\fC}{\mathfrak{c}}
\nwc{\fD}{\mathfrak{d}}
\nwc{\fE}{\mathfrak{e}}
\nwc{\fF}{\mathfrak{f}}
\nwc{\fG}{\mathfrak{g}}
\nwc{\fH}{\mathfrak{h}}
\nwc{\fI}{\mathfrak{i}}
\nwc{\fJ}{\mathfrak{j}}
\nwc{\fK}{\mathfrak{k}}
\nwc{\fL}{\mathfrak{l}}
\nwc{\fM}{\mathfrak{m}}
\nwc{\fN}{\mathfrak{n}}
\nwc{\fO}{\mathfrak{o}}
\nwc{\fP}{\mathfrak{p}}
\nwc{\fQ}{\mathfrak{q}}
\nwc{\fR}{\mathfrak{r}}
\nwc{\fS}{\mathfrak{s}}
\nwc{\fT}{\mathfrak{t}}
\nwc{\fU}{\mathfrak{u}}
\nwc{\fV}{\mathfrak{v}}
\nwc{\fW}{\mathfrak{w}}
\nwc{\fX}{\mathfrak{x}}
\nwc{\fY}{\mathfrak{y}}
\nwc{\fZ}{\mathfrak{z}}

\nwc{\tA}{\widetilde{A}}
\nwc{\tB}{\widetilde{B}}
\nwc{\tE}{E^{\vareps}}
\nwc{\tk}{\tilde k}
\nwc{\tN}{\tilde N}
\nwc{\tP}{\widetilde{P}}
\nwc{\tQ}{\widetilde{Q}}
\nwc{\tR}{\widetilde{R}}
\nwc{\tV}{\widetilde{V}}
\nwc{\tW}{\widetilde{W}}
\nwc{\ty}{\tilde y}
\nwc{\teta}{\tilde \eta}
\nwc{\tdelta}{\tilde \delta}
\nwc{\tlambda}{\tilde \lambda}
\nwc{\ttheta}{\tilde \theta}
\nwc{\tvartheta}{\tilde \vartheta}
\nwc{\tPhi}{\widetilde \Phi}
\nwc{\tpsi}{\tilde \psi}
\nwc{\tmu}{\tilde \mu}

\nwc{\To}{\longrightarrow} 

\nwc{\ad}{\rm ad}
\nwc{\eps}{\epsilon}
\nwc{\ep}{\epsilon}
\nwc{\vareps}{\varepsilon}

\def\ep{\epsilon}
\def\tr{{\rm tr}}

\def\sq2{\sqrt{2}}

\def\t2{{\mathbb T}^2}
\def\s2{{\mathbb S}^2}

\def\R{\mathbb{R}}

\def\Z{\mathbb{Z}}
\def\C{\mathbb{C}}

\nwc{\lap}{\bigtriangleup}
\nwc{\rest}{\restriction}
\nwc{\Diff}{\operatorname{Diff}}
\nwc{\diam}{\operatorname{diam}}
\nwc{\Res}{\operatorname{Res}}
\nwc{\Spec}{\operatorname{Spec}}
\nwc{\Vol}{\operatorname{Vol}}
\nwc{\Op}{\operatorname{Op}}
\nwc{\supp}{\operatorname{supp}}
\nwc{\Span}{\operatorname{span}}

\nwc{\dia}{\varepsilon}
\nwc{\cut}{f}
\nwc{\qm}{u_\hbar}

\def\hto0{\xrightarrow{\hbar\to 0}}

\def\rto0{\xrightarrow{r\to 0}}

\nwc{\la}{\langle}
\nwc{\ra}{\rangle}
\nwc{\lp}{\left(}
\nwc{\rp}{\right)}

\nwc{\bequ}{\begin{equation}}
\nwc{\be}{\begin{equation}}
\nwc{\ben}{\begin{equation*}}
\nwc{\bea}{\begin{eqnarray}}
\nwc{\bean}{\begin{eqnarray*}}
\nwc{\bit}{\begin{itemize}}
\nwc{\bver}{\begin{verbatim}}

\nwc{\eequ}{\end{equation}}
\nwc{\ee}{\end{equation}}
\nwc{\een}{\end{equation*}}
\nwc{\eea}{\end{eqnarray}}
\nwc{\eean}{\end{eqnarray*}}
\nwc{\eit}{\end{itemize}}
\nwc{\ever}{\end{verbatim}}